\documentclass[article,a4paper,11pt]{amsart}
\usepackage[utf8]{inputenc}
\usepackage[active]{srcltx}
\usepackage{verbatim}

\usepackage{esint,graphicx,color,epsfig,latexsym,longtable,amsmath,amscd,amsxtra,amssymb,amsxtra,exscale,xy,mathrsfs,wrapfig}

\usepackage{hyperref}

\usepackage[fleqn,tbtags]{mathtools}
\mathtoolsset{showonlyrefs,showmanualtags}

\theoremstyle{plain}
\newtheorem{theorem}{Theorem}[section]

\newtheorem{proposition}[theorem]{Proposition}
\theoremstyle{remark}

\newtheorem{hypothesis}[theorem]{Hypothesis}

\newtheorem{remark}[theorem]{Remark}

 \numberwithin{equation}{section}

\begin{document}

\allowdisplaybreaks

\title[Dirichlet-to-Neumann semigroup for an eigenvalue problem]
{Dirichlet-to-Neumann semigroup with respect to a general second order eigenvalue problem}

\author{Jamil Abreu}

\address{Departamento de Matem\'atica\\
Universidade Federal de S\~ao Carlos, Rodovia Washington Lu\'is \\ S\~ao Carlos SP \\ Brazil}
\email{jamil@dm.ufscar.br}

\author{\'Erika Capelato}

\address{Departamento de Economia\\
Universidade Estadual J\'ulio de Mesquita, Rodovia Araraquara-Ja\'u, km 1 \\ Araraquara SP \\ Brazil}
\email{erika@fclar.unesp.br}




\date\today

\begin{abstract}
In this paper we study a Dirichlet-to-Neumann operator with respect to a second order elliptic operator with measurable 
coefficients, including first order terms, namely, the operator on $L^2(\partial\Omega)$ given by $\varphi\mapsto \partial_{\nu}u$ 
where $u$ is a weak solution of 
\begin{equation}
\left\{
\begin{aligned}
-{\rm div}\, (a\nabla u) +b\cdot \nabla u -{\rm div}\, (cu)+du & =\lambda u   \ \ \text{on}\ \Omega,\\
u|_{\partial\Omega}    & =\varphi .
\end{aligned}
\right.
\end{equation}
Under suitable assumptions on the matrix-valued function $a$, on the vector fields $b$ and $c$, and on the function $d$, we investigate positivity, 
sub-Markovianity, irreducibility and domination properties of the associated Dirichlet-to-Neumann semigroups. 
\end{abstract}

\maketitle

\section{Introduction}\label{sec:intro}

Form methods for evolution equations date back to the pioneering works by D. Hilbert on integral equations in the beginning of the 
twentieth century. However, it was not until the late 1950s that they have been systematically developed towards its applications to evolution 
equations. At this early period two schools have emerged, one centered around J.-L. Lions (elliptic forms) and other around T. Kato 
(sectorial forms). Both notions (elliptic and sectorial forms) turn out to be essentially equivalent, being different descriptions of 
the same ideas. Standard references for both theories include R. Dautray and J.-L. Lions's book \cite{dautray1992mathematical-5} and 
T. Kato's book \cite{kato1976perturbation}; more recent developments have been documented in E.-M. Ouhabaz's book 
\cite{ouhabaz2005analysis}. 

In a recent paper, W. Arendt and A.F.M. ter Elst \cite{arendt2012sectorial} have extended the classical form method in many respects. In the case 
these authors call the `complete case', which corresponds to Lions's elliptic forms, the 
form domain $V$ is allowed to be a Hilbert space (over $\mathbb{K}=\mathbb{R}$ or $\mathbb{C}$) not necessarily embedded in the reference space, say $H$, 
provided there is a bounded linear operator $j:V\to H$ with dense range; if $\mathfrak{a}:V\times V\to \mathbb{K}$ is a continuous sesquilinear form 
which is $j$-elliptic in the sense that 
\begin{equation}\label{eq:form-j_elliptic}
{\rm Re}\, \mathfrak{a}(u,u)+\omega\|j(u)\|_H^2\geqslant \alpha\|u\|_V^2\ \ \ (u\in V)
\end{equation}
for some constants $\omega\in \mathbb{R}$ and $\alpha>0$, then an operator $A$ on $H$ can be associated to $\mathfrak{a}$ in such a way that 
\begin{multline}\label{eq:associated_operator}
x\in \mathscr{D}(A) \ \text{and} \ Ax=f \ \text{if, and only if}\\
x=j(u) \ \text{for some} \ u\in V \ \text{and} \ \mathfrak{a}(u,v)=(f|j(v))_H \ \text{for all}\  v\in V.
\end{multline}
A further consequence for the so called `incomplete case', which corresponds to Kato's sectorial forms, is that an $m$-sectorial operator 
(and therefore, a holomorphic semigroup generator) can be associated to a densely defined sectorial form, regardless it is closable or not. 

This new form method allows an elegant treatment of the so-called Dirichlet-to-Neumann operator. Let $\Omega\subset\mathbb{R}^N$ be a bounded open 
set with Lipschitz boundary. By definition, the Dirichlet-to-Neumann operator is the 
operator $D_0$ acting on $L^2(\partial\Omega)$ with the property that $\varphi\in \mathscr{D}(D_0)$ and 
$D_0\varphi =h$ if, and only if there is a weak solution $u\in H^1(\Omega)$  of
\begin{equation}
\left\{
\begin{aligned}
\Delta u & =0   \ \ \text{on}\ \Omega,\\
u|_{\partial\Omega}    & =\varphi ,
\end{aligned}
\right.
\end{equation}
such that $\partial_\nu u =h$ in a weak sense; an element $u\in H^1(\Omega)$ with distributional Laplacian $\Delta u\in L^2(\Omega)$ is said to have a 
weak normal derivative if there exists $h\in L^2(\partial\Omega)$ such that Green's formula holds, meaning that the identity
\begin{equation}\label{eq:derivada_normal_fraca}
\int_{\Omega}(\Delta u)\overline{v}\, dx+\int_{\Omega}\nabla u\cdot \overline{\nabla v}\, dx=\int_{\partial\Omega}h \overline{v}\, d\sigma
\end{equation}
holds for every $v\in H^1(\Omega)$. In this case we set $\partial_\nu u:=h$. By showing that $D_0$ is associated with a $j$-elliptic form, namely, 
the classical Dirichlet form 
$$
\mathfrak{a}(u,v)=\int_\Omega \nabla u\cdot \overline{\nabla v}\, dx\ \ \ (u,v\in H^1(\Omega))
$$
with $j:H^1(\Omega)\to L^2(\partial\Omega)$ being the trace, Arendt \& ter Elst have provided an interesting application of their theory where a 
non-injective $j$ appears in a natural way.

In this paper we study the Dirichlet-to-Neumann operator, to be denoted by $D_\lambda^{\mathscr{A}}$, defined by 
$\varphi\mapsto \partial_\nu u$ where $u\in H^1(\Omega)$ is a weak solution of the eigenvalue problem
\begin{equation}\label{eq:eigenvalue-problem-DtN-perturb}
\left\{
\begin{aligned}
-{\rm div}\, (a\nabla u) +b\cdot \nabla u -{\rm div}\, (cu)+du & =\lambda u   \ \ \text{on}\ \Omega,\\
u|_{\partial\Omega}    & =\varphi ,
\end{aligned}
\right.
\end{equation}
and $\partial_\nu u$ is the `weak conormal derivative' which, in the smooth case, coincides with the classical conormal derivative 
$(a\nabla u +cu)\cdot \nu$. In Problem \eqref{eq:eigenvalue-problem-DtN-perturb}, $a$ is a matrix-valued function, $b$ and $c$ are vector fields, 
$d$ is measurable function and $\lambda$ is a number; for the precise hypotheses on these data, see Theorem \ref{thm:main} below. The 
difficulty here lies, of course, in the presence of the first order terms `$b\cdot \nabla u$' and 
`${\rm div}\, (cu)$'. To the best of our knowledge, Dirichlet-to-Neumann operators whose internal dynamics includes first order terms have been 
first considered in \cite{behrndt2012inverse} in connection with Calder\'on's inverse problem which asks, roughly speaking, 
whether $\mathscr{A}$ can be determined from $D_\lambda^{\mathscr{A}}$. Following Arendt \& ter Elst approach to the Dirichlet-to-Neumann 
operator $D_0$ through form methods it is clear that $D_\lambda^{\mathscr{A}}$ should be, at best, 
the associated operator, in the sense of \eqref{eq:associated_operator}, to the sesquilinear form 
$\mathfrak{a}_\lambda:H^1(\Omega)\times H^1(\Omega)\to \mathbb{K}$ defined by
\begin{equation}\label{eq:form-dtn-perturb}
\mathfrak{a}_\lambda(u,v)=\int_\Omega a\nabla u\cdot \overline{\nabla v}\, dx+\int_\Omega (b\cdot\nabla u)\overline{v}\, dx
+\int_\Omega u (c\cdot \overline{\nabla v})\, dx+\int_\Omega du\overline{v}\, dx-\lambda\int_\Omega u\overline{ v}\, dx.
\end{equation}
However, as shown in \cite{arendt2012sectorial} (cf. also \cite{arendt2007spectral}), this form is not $j$-elliptic in general, even when $a=I$, 
$b=c=0$ and $d=0$. This lack of ellipticity can be circumvented by a general procedure. Roughly speaking, to any sesquilinear form 
$\mathfrak{a}:V\times V\to \mathbb{K}$ an 
m-sectorial operator $A$ can still be associated to $\frak{a}$ in the sense of \eqref{eq:associated_operator} provided $\frak{a}$ is $j$-elliptic on a 
suitable closed subspace of $V$ which complements $\mathscr{N}(j)$; the precise statement will be recalled below in Proposition 
\ref{prop:associated_operator-V-kerj-Vj}. Moreover, the recent theory of `compactly elliptic forms' introduced in Arendt et al 
\cite{arendt2014dirichlet} makes this task even easier and we briefly describe how this theory can also be used in the construction of 
our Dirichlet-to-Neumann operator.

The present work is motivated by some results in \cite{arendt2012friedlander} (cf. also \cite{arendt2007spectral}) where it has been shown, among other 
things, that the semigroup generated by the Dirichlet-to-Neumann operator with respect to the eigenvalue problem 
\begin{equation}\label{eq:eigenvalue-problem-DtN}
\left\{
\begin{aligned}
-\Delta u  & =\lambda u   \ \ \text{on}\ \Omega,\\
u|_{\partial\Omega}    & =\varphi ,
\end{aligned}
\right.
\end{equation}
which corresponds to Problem \eqref{eq:eigenvalue-problem-DtN-perturb} with $a=I$, $b=c=0$ and $d=0$, is positive and irreducible whenever $\lambda<\lambda_1^\text{D}$, 
$\lambda_1^\text{D}$ being the first eigenvalue of the Dirichlet Laplacian given in variational terms by 
\begin{equation}\label{eq:first_eigenvalue-D_Laplacian}
\lambda_1^\text{D}=\inf_{u\in H^1_0(\Omega),u\neq 0}\frac{\int_\Omega |\nabla u|^2\, dx}{\int_\Omega |u^2|\, dx}.
\end{equation}
The question whether this semigroup remains positive or not for 
$\lambda>\lambda_1^\text{D}$ is a major research topic; recent contributions to this issue include e.g. the paper by D. Daners \cite{daners2014non}. 
Here we do not address this question but focus on the problem whether some properties, having positivity at their center, 
of the Dirichlet-to-Neumann semigroup is preserved under first order perturbations of Problem \eqref{eq:eigenvalue-problem-DtN}. With little 
additional effort we can also describe when these semigroups are in fact sub-Markovian. Moreover, we also consider irreducibility and 
some domination properties. Some of these questions have been also studied, in connection with Calder\'on's problem, in \cite{ouhabaz2016milder}. 

Let $e^{-tD_\lambda^\mathscr{A}}$ be the semigroup on $L^2(\partial\Omega)$ generated by $-D_\lambda^\mathscr{A}$. 
In the following, $A_\text{D}$ denotes the realization of $\mathscr{A}$ with Dirichlet boundary conditions (see next section). For 
simplicity, we also consider real scalars.

\begin{theorem}\label{thm:main}
Let $\Omega\subset \mathbb{R}^N$ be a bounded connected open set with Lipschitz boundary. Suppose the matrix-valued 
function $a\in L^\infty(\Omega ;\mathbb{R}^{N\times N})$ is symmetric and uniformily positive-definite in the sense that, for some $\kappa>0$,
\begin{equation}\label{eq:elipticity_condition-matrix}
a(x)\xi \cdot\xi\geqslant \kappa |\xi|^2\ \ \ (\xi\in \mathbb{R}^N, \text{a.e.}\ x\in \Omega).
\end{equation}
Suppose the vector fields $b,c\in C^1(\overline{\Omega})^N$ are real and satisfy ${\rm div}\, b ={\rm div}\, c= 0$ and 
$b\cdot \nu= c\cdot \nu= 0$. Let $d\in L^\infty(\Omega)$ be real-valued. Suppose $\lambda\in \mathbb{R}\backslash\sigma(A_\text{D})$.
\begin{enumerate}
\item\label{item:semigroup-DtN-positive} If 
$4\kappa^{-1}\|b-c\|^2_{L^\infty(\Omega)^N}+\|d^-\|_{L^\infty(\Omega)}+\lambda <\kappa\lambda_1^\text{D}$ then $e^{-tD_\lambda^\mathscr{A}}$ 
is positive.
\item\label{item:semigroup-DtN-submarkovian} If 
$4\kappa^{-1}\|b-c\|^2_{L^\infty(\Omega)^N}+\|d^-\|_{L^\infty(\Omega)}+\lambda <\kappa\lambda_1^\text{D}$ and 
$\lambda\leqslant d$ then $e^{-tD_\lambda^\mathscr{A}}$ is sub-Markovian.
\item\label{item:semigroup-DtN-irreducible} If $\|d^-\|_{L^\infty(\Omega)}+\lambda <\kappa\lambda_1^\text{D}$ then $e^{-tD_\lambda^\mathscr{A}}$ is 
irreducible.
\item\label{item:semigroup-DtN-domination} If $b=c$ and $\lambda_2\leqslant \lambda_1<\kappa\lambda_1^\text{D}-\|d^-\|_{L^\infty(\Omega)}$ then 
$0\leqslant e^{-tD_{\lambda_2}^\mathscr{A}}\leqslant e^{-tD_{\lambda_1}^\mathscr{A}}$ in the sense of positive operators, i.e.
$$
0\leqslant e^{-tD_{\lambda_2}^\mathscr{A}}\varphi\leqslant e^{-tD_{\lambda_1}^\mathscr{A}}\varphi\ \ \ (t>0, \ 0\leqslant \varphi\in L^2(\partial\Omega)).
$$
\end{enumerate}
\end{theorem}

The regularity required on the boundary $\partial\Omega$ has two purposes: first, it guarantees that elements in $H^1(\Omega)$ have a well-defined trace on 
the boundary and the trace operator $j:H^1(\Omega)\to L^2(\partial\Omega)$ is compact; second, the divergence theorem holds, that is, there is 
an outward unit normal $\nu$ defined a.e. on $\partial\Omega$ and 
\[
\int_\Omega \partial_j u\, dx =\int_{\partial\Omega}u\nu_j\, d\sigma\ \ (u\in H^1(\Omega)).
\]

Let us finish this introduction by briefly describing how the paper is organized. In Section \ref{sec:prelim} we define 
realizations of a second order operator $\mathscr{A}=-{\rm div}\, (a\nabla u) +b\cdot \nabla u -{\rm div}\, (cu)+du$ under various boundary 
conditions, which will play in Problem \eqref{eq:eigenvalue-problem-DtN-perturb} the same role as the Laplacian does 
in Problem \eqref{eq:eigenvalue-problem-DtN}. We also recall the basic definitions and relevant properties of positive, irreducible and 
sub-Markovian semigroups which are needed in the sequel. In Section \ref{sec:DtN-operator} we define the main object of study here, namely, the 
Dirichlet-to-Neumann operator with respect to Problem \eqref{eq:eigenvalue-problem-DtN-perturb}, denoted by $D_\lambda^{\mathscr{A}}$, and prove the 
analogous version of the folklore result which relates its spectrum to the spectrum of the realization of $\mathscr{A}$ with Robin boundary 
conditions. The proof of Theorem \ref{thm:main} is carried out in Section \ref{sec:DtN-semigroup}.

\section{Preliminaries}\label{sec:prelim}

Let us start by formulating the following hypothesis which we assume throughout this section. Let $\Omega\subset \mathbb{R}^N$ be an open set. Let 
$\mathbb{K}$ be either $\mathbb{R}$ or $\mathbb{C}$. 

\begin{hypothesis}\label{hypothesis}
The matrix-valued function $a\in L^\infty(\Omega ;\mathbb{C}^{N\times N})$ is Hermitian and uniformly positive-definite in the sense that
$$
\text{Re}\, a(x)\xi \cdot\overline{\xi}\geqslant \kappa |\xi|^2\ \ \ (\xi\in \mathbb{C}^N, \text{a.e.}\ x\in \Omega).
$$
The vector fields $b,c\in L^\infty(\Omega)^N$ as well the measurable function $d\in L^\infty(\Omega)$ are, possibly, complex-valued.
\end{hypothesis}

We will consider suitable realizations of the second order elliptic operator in divergence form
\begin{equation}
\begin{aligned}\label{eq:operador_eliptico}
\mathscr{A}u&=\sum_{j,k=1}^N-\partial_{j}(a_{jk}\partial_ku)+\sum_{j=1}^Nb_j\partial_ju-\sum_{j=1}^N\partial_j(c_ju)+du\\
&=-{\rm div}\, (a\nabla u) +b\cdot \nabla u -{\rm div}\, (cu)+du
\end{aligned} 
\end{equation}
on $L^2(\Omega)$ under Dirichlet and Robin boundary conditions; such realizations play the role of the Dirichlet and Robin Laplacian as we pass from 
the Dirichlet-to-Neumann operator $D_\lambda$ with respect to Problem \eqref{eq:eigenvalue-problem-DtN} to the one relative to 
Problem \eqref{eq:eigenvalue-problem-DtN-perturb}. Clearly the associated form is
\begin{equation}
\begin{aligned}\label{eq:associated_form}
\mathfrak{a}(u,v)&=\int_{\Omega} \Big(\sum_{j,k=1}^Na_{jk}\partial_ku\overline{\partial_{j}v}+\sum_{j=1}^Nb_j\partial_ju\overline{v}
+\sum_{j=1}^Nc_ju\overline{\partial_{j}v}+du\overline{v}\Big)dx\\
&=\int_\Omega (a\nabla u)\cdot \overline{\nabla v}\, dx+\int_\Omega (b\cdot\nabla u)\overline{v}\, dx
+\int_\Omega u (c\cdot \overline{\nabla v})\, dx+\int_\Omega du\overline{v}\, dx
\end{aligned}
\end{equation}
with domain $H^1(\Omega)$. For an element $u\in H^1(\Omega)$ we say that $\mathscr{A}u\in L^2(\Omega)$ if there exists an element $f\in L^2(\Omega)$, 
in which case we write $\mathscr{A}u=f$, such that 
\begin{equation}\label{eq:condicao-operador_maximal}
\int_{\Omega}f\overline{v}\, dx =\frak{a}(u,v)\ \ \ (v\in C_\text{c}^{\infty}(\Omega)).
\end{equation}

\begin{proposition}\label{prop:operator-delta-plus-nabla-H10}
Let $\Omega\subset \mathbb{R}^N$ be an open set and assume Hypothesis \ref{hypothesis}. Let $A_\text{D}$ be the operator on $L^2(\Omega)$ defined by
\begin{align*}
\mathscr{D}(A_\text{D}) & :=\{u\in H_0^1(\Omega):\mathscr{A}u \in L^2(\Omega)\},\\
A_\text{D}u & :=\mathscr{A}u\ \ \ (u\in \mathscr{D}(A_\text{D})).
\end{align*}
Then $-A_\text{D}$ is the generator of a quasi-contrative $C_0$-semigroup. If $\mathbb{K}=\mathbb{C}$ then $-A_\text{D}$ generates a cosine 
operator function and hence a holomorphic semigroup of angle $\pi/2$.
\end{proposition}

Before we go into the proof we quote the following result which has been noted in \cite{mugnolo2012convergence}. We write 
$\mathfrak{a}(u)$ for $\mathfrak{a}(u,u)$ throughout this paper.

\begin{proposition}\label{proposition:Mugnolo-Nittka-JEE2012-numerical}
Let $V$ and $H$ be Hilbert spaces and let $j:V\to H$ be a bounded linear operator with dense range. Let $\mathfrak{a}:V\times V\to \mathbb{C}$ 
be a $j$-elliptic form with associated operator $A$ on $H$. If there exists an $M\geqslant 0$ such that 
\begin{equation}\label{eq:estimate-imaginary-part-form}
|{\rm Im}\, \frak{a}(u)|\leqslant M\|u\|_V\|j(u)\|_H\ \ (u\in V)
\end{equation}
then $-A$ generates a cosine operator function and hence a holomorphic semigroup of angle $\pi/2$.
\end{proposition}

In fact, estimate \eqref{eq:estimate-imaginary-part-form} implies that the numerical range of $A$ lies in a parabola with vertex on the real axis 
and opened in the direction of the positive real axis. Thus the assertion that $-A$ generates a cosine operator function follows from a theorem due 
to M. Crouzeix. Moreover, it is known that every generator of a cosine operator function also generates a holomorphic semigroup 
of angle $\pi/2$; for more details and references, see \cite[Proposition 2.4]{mugnolo2012convergence}. 

\begin{proof}[Proof of Proposition \ref{prop:operator-delta-plus-nabla-H10}]
Let $\frak{a}:H_0^1(\Omega)\times H_0^1(\Omega)\to \mathbb{K}$ be the sesquilinear form defined by the same formula as in Eq. \eqref{eq:associated_form}, 
but restricted to $H_0^1(\Omega)$. Recall that, under Hypothesis \ref{hypothesis}, $b,c\in L^\infty(\Omega)^N$. We first estimate
\begin{equation}\label{eq:estimate-re-intb-gradu-u-2}
\begin{aligned}
{\rm Re}\, \int_\Omega ((b+\overline{c})\cdot\nabla u)\overline{u}\, dx&\geqslant 
-\|b+\overline{c}\|_{L^\infty(\Omega)^N}\|\nabla u\|_{L^2(\Omega)^N}\|u\|_{L^{2}(\Omega)}\\
&\geqslant -\|b+\overline{c}\|_{L^\infty(\Omega)^N}(\varepsilon\|\nabla u\|_{L^2(\Omega)^N}^2+c_\varepsilon\|u\|_{L^{2}(\Omega)}^2).
\end{aligned} 
\end{equation}
If we choose $\varepsilon>0$ such that $\|b+\overline{c}\|_{L^\infty(\Omega)^N}\varepsilon =\frac{\kappa}{2}$ then we get the estimate
\begin{equation}
{\rm Re}\, \frak{a}(u)+(\|b+\overline{c}\|_{L^\infty(\Omega)^N}c_\varepsilon +\|({\rm Re}\, d)^-\|_{L^\infty(\Omega)}) \|u\|_{L^2(\Omega)}^2 
\geqslant \frac{\kappa}{2}\int_\Omega |\nabla u|^2\, dx .
\end{equation}
From this, with $\omega_1=\|b+\overline{c}\|_{L^\infty(\Omega)^N}c_\varepsilon +\|({\rm Re}\, d)^-\|_{L^\infty(\Omega)} +\frac{\kappa}{2}$, we obtain
\begin{equation}\label{eq:estimate-form-elliptic-H10}
{\rm Re}\, \frak{a}(u)+\omega_1 \|u\|_{L^2(\Omega)}^2 \geqslant \frac{\kappa}{2}\|u\|_{H^1(\Omega)}^2\ \ \ (u\in H^1_0(\Omega)). 
\end{equation}
Thus, $\frak{a}$ is $L^2(\Omega)$-elliptic. It is elementary to check that $A_\text{D}$ is the operator associated 
with $\frak{a}$; thus the assertion that it generates a quasi-contrative $C_0$-semigroup follows from the general 
theory, cf. e.g \cite[Theorem 5.7]{arendt-voigt:18ISEM}.

Moreover, 
\begin{align}
|{\rm Im}\, \frak{a}(u)|&=
\Big|{\rm Im}\, \Big(\int_\Omega (b\cdot\nabla u)\overline{u}\, dx+\int_\Omega u (c\cdot \overline{\nabla u})\, dx+\int_\Omega d|u|^2\, dx\Big)\Big|\\
&\leqslant (\|b-\overline{c}\|_{L^\infty(\Omega)^N}+\|{\rm Im}\, d\|_{L^\infty(\Omega)})\|u\|_{H^1(\Omega)}\|u\|_{L^2(\Omega)},
\end{align}
thus the last assertion follows from Proposition \ref{proposition:Mugnolo-Nittka-JEE2012-numerical}.
\end{proof}

\begin{remark}\label{remark:basic_estimate}
Note that in the derivation of estimate \eqref{eq:estimate-form-elliptic-H10} no special property of $H_0^1(\Omega)$ is used, so that it is 
still valid for elements $u\in H^1(\Omega)$ whenever the form $\mathfrak{a}$ defined by Eq. \eqref{eq:associated_form} is considered on 
$H^1(\Omega)$. We use this in the following without further ado.  
\end{remark}

Let $\Omega\subset \mathbb{R}^N$ be a bounded open set with Lipschitz boundary. In this case elements in $H^1(\Omega)$ have a well-defined trace 
on the boundary. We say an element $u\in H^1(\Omega)$ with $\mathscr{A}u\in L^2(\Omega)$ has a weak conormal derivative if there 
exists $h\in L^2(\partial\Omega)$ such that
\begin{equation}\label{eq:derivada_conormal_fraca}
-\int_{\Omega}(\mathscr{A}u)\overline{v}\, dx+\frak{a}(u,v)=\int_{\partial\Omega}h \overline{v}\, d\sigma
\end{equation}
holds for every $v\in H^1(\Omega)$. In this case we put $\partial_\nu u:=h$; this definition is natural in the sense that it reduces to the 
classical notion of conormal derivative (smooth case) and also to the definition of weak normal derivative introduced in \cite{arendt2012sectorial} 
(see also \cite{arendt2011dirichlet} and \cite{arendt2014dirichlet}). By repeating the same proof above we can define a realization of 
$\mathscr{A}$ with Neumann boundary conditions, namely, an operator $A_\text{N}$ on $L^2(\Omega)$ given by
\begin{align*}
\mathscr{D}(A_\text{N}) & :=\{u\in H^1(\Omega):\mathscr{A}u \in L^2(\Omega), \ 
\text{`$\partial_\nu u=0$'}\},\\
A_\text{N}u & :=\mathscr{A}u\ \ \ (u\in \mathscr{D}(A_\text{N})),
\end{align*}
where `$\partial_\nu u=0$' means `$\frak{a}(u,v)=\int_{\Omega}(\mathscr{A}u)\overline{v}\, dx$ for all $v\in H^1(\Omega)$', or, equivalently, that the 
weak conormal derivative exists and equals zero. The operator $A_\text{N}$, however, will not be relevant in this paper.

Next we consider a realization of $\mathscr{A}$ with the Robin boundary condition $\partial_{\nu}u+\beta u|_{\partial\Omega}=0$, 
where $\beta\in L^\infty(\partial\Omega)$. Let $\mathfrak{a}^{\beta}:H^1(\Omega)\times H^1(\Omega)\to \mathbb{K}$ be the 
sesquilinear form defined by
\begin{equation}\label{eq:forma-Robin-beta}
\mathfrak{a}^{\beta}(u,v)=\mathfrak{a}(u,v)+\int_{\partial\Omega} \beta u\overline{v}\, d\sigma . 
\end{equation}
Since, under our present assumptions on $\Omega$, the trace is compact from $H^1(\Omega)$ to $L^2(\partial\Omega)$, it follows from 
Lions's lemma \cite[Ch. 2, Lemma 6.1]{necas2012direct} that there exists $c_1\geqslant 0$ such that
\begin{align*}
\|({\rm Re}\, \beta)^-\|_{L^{\infty}(\partial\Omega)}\int_{\partial\Omega}|u|^2\, d\sigma&\leqslant \frac{\kappa}{4}\|u\|_{H^1(\Omega)}^2+c_1\int_\Omega |u|^2\, dx ,
\end{align*}
thus from estimate \eqref{eq:estimate-form-elliptic-H10} (cf. Remark \ref{remark:basic_estimate}) we get
\begin{equation}
 {\rm Re}\, \frak{a}^\beta(u)+ (\omega_1 +c_1) \|u\|_{L^2(\Omega)}^2 \geqslant \frac{\kappa}{4}\|u\|_{H^1(\Omega)}^2\ \ \ (u\in H^1(\Omega)).
\end{equation}
Therefore, $\mathfrak{a}^\beta$ is $L^2(\Omega)$-elliptic. Moreover, by a standard trace inequality the form $\mathfrak{a}^\beta$ 
clearly satisfies an estimate of the form \eqref{eq:estimate-imaginary-part-form}. We have thus proved the following.

\begin{proposition}\label{prop:operator-A-Robin_condition}
Let $\Omega\subset \mathbb{R}^N$ be a bounded open set with Lipschitz boundary, let $\beta\in L^\infty(\partial\Omega)$ and assume 
Hypothesis \ref{hypothesis}. Let $A_\beta$ be the operator on $L^2(\Omega)$ defined by
\begin{align*}
\mathscr{D}(A_\beta) & :=\{u\in H^1(\Omega): \mathscr{A}u \in L^2(\Omega),\ 
\partial_{\nu}u+\beta u|_{\partial\Omega}=0 \},\\
A_\beta u & :=\mathscr{A}u\ \ \ (u\in \mathscr{D}(A_\beta)).
\end{align*}
Then $-A_\beta$ is the generator of a quasi-contrative $C_0$-semigroup. If $\mathbb{K}=\mathbb{C}$ then $-A_\beta$ generates a cosine operator function 
and hence a holomorphic semigroup of angle $\pi/2$.
\end{proposition}

Let $(\Omega,\mathscr{A},\mu)$ be a measure space and $1\leqslant p<\infty$. A semigroup $T$ on $L^p(\Omega ,\mu)$ is positive if 
$T(t)f\geqslant 0$ whenever $f\geqslant 0$ and $t>0$. A semigroup $T$ on $L^p(\Omega ,\mu)$ is sub-Markovian if, besides being positive, it is also 
$L^\infty$-contractive, meaning that $\|T(t)f\|_{\infty}\leqslant \|f\|_\infty$ for all $f\in L^p\cap L^\infty(\Omega)$. This property 
is crucial in connection with the problem of extrapolating the semigroup to the $L^p$ scale which is the starting point for further investigations of 
their spectral properties. In order to establish these properties, which are equivalent to the invariance of certain convex and closed sets, we 
employ the following $j$-elliptic version of Ouhabaz's invariance theorem, proved in \cite[Proposition 2.9]{arendt2012sectorial}. Actually, 
the version stated and proved in \cite[Proposition 2.9]{arendt2012sectorial} assumes in addition that the form is accretive and the 
consequence of this is that item (c) below is possible with $\omega=0$. This may be interesting if one is concerned with item (c) as a necessary condition 
but here we are interested in the implication `(c)$\Rightarrow$(a)' so that the version stated below, which can be proven by adapting the 
proof in \cite[Theorem 9.20]{arendt-voigt:18ISEM}, is more convenient. Moreover, $P:H\to C$ is the minimizing 
projection, see e.g. \cite[Theorem 5.2]{brezis2011functional}.

\begin{proposition}\label{prop:Ouhabaz-invariance_criteria}
Let $V$ and $H$ be Hilbert spaces and let $j:V\to H$ be a bounded linear operator with dense range. Let 
$\mathfrak{a}:V\times V\rightarrow \mathbb{K}$ be a $j$-elliptic and continuous sesquilinear form with associated operator $A$. Let $T$ 
be the semigroup generated by $-A$. Let $C\subset H$ be a non-empty closed convex set with minimizing projection $P:H\rightarrow C$. Then 
the following assertions are equivalent. 
\begin{enumerate}
\item $C$ is invariant under $T$.
\item For all $u\in V$ there exists $w\in V$ such that  
\begin{equation}\label{eq:condicao-invariancia-forma-versao_geral}
 Pj(u)=j(w)\ \ \text{and}\ \ {\rm Re}\, \mathfrak{a}(w,u-w)\geqslant 0.
\end{equation}
\item For all $u\in V$ there exists $w\in V$ such that  
\begin{equation*}
 Pj(u)=j(w)\ \ \text{and}\ \ {\rm Re}\, \mathfrak{a}(u,u-w)\geqslant -\omega\|j(u)-j(w)\|_H^2
\end{equation*}
for some $\omega\in \mathbb{R}$ depending only on the form $\mathfrak{a}$.
\end{enumerate}
\end{proposition}

The celebrated Krein-Rutman theorem asserts that if the generator $A$ 
of a positive semigroup has compact resolvent and $s(A)>-\infty$ ($s(A)$ is the spectral bound of $A$), then $-s(A)$ is the first 
eigenvalue $\lambda_1(-A)$ of $-A$ and admits a positive eigenfunction. We refer the interested 
reader to \cite[Lecture 10]{arendt:9ISEM} for more information. 

Let $(\Omega,\mathscr{A},\mu)$ be a measure space and $1\leqslant p<\infty$. By definition, a semigroup $T$ on $L^p(\Omega ,\mu)$ is called 
irreducible if the only closed ideals $\mathscr{I}\subset L^p(\Omega ,\mu)$ (they are necessarily of the form $L^p(\omega)$, for some 
$\omega\in \mathscr{A}$) which are invariant by $T$ (in the sense that $T\mathscr{I}\subset \mathscr{I}$) are $\mathscr{I}=\{0\}$ and 
$\mathscr{I}=L^p(\Omega ,\mu)$ itself. The following well known result is often used to deduce irreducibility for a large class of 
elliptic operators. See Remark \ref{remark:irreducibility}.

\begin{proposition}\label{proposition:irreducibility}
Let $\Omega\subset \mathbb{R}^N$ be an open connected set. Let $T$ be a strongly continuous semigroup on $L^2(\Omega)$ associated to an elliptic form 
$\mathfrak{a}:V\times V\to \mathbb{K}$, where $V\subset H^1(\Omega)$ is a subspace containing $C_\text{c}^\infty(\Omega)$. Then $T$ is irreducible.
\end{proposition}

\begin{proof}
Suppose $T$ leaves $L^2(\omega)$ invariant for some measurable $\omega\subset \Omega$. By Ouhabaz invariance theorem, 
$\chi_\omega u\in V$ whenever $u\in V$. As in the proof of \cite[Theorem 4.5]{ouhabaz2005analysis} 
(cf. also \cite[Proposition 11.1.2]{arendt:9ISEM}), it follows from connectedness that either $\omega$ has measure zero or 
$\Omega\backslash\omega$ has measure zero; this is the hard part of the proof but it is well known (see the references just given). Thus, 
either $L^2(\omega)=\{0\}$ or $L^2(\omega)=L^2(\Omega)$.
\end{proof}

\begin{remark}\label{remark:irreducibility}
The statement and proof of Proposition \ref{proposition:irreducibility} may be surprising to some readers, due to its simplicity. It is 
convenient to say some words about this. Experts know very well that related results on irreducibility as stated in Ouhabaz's book 
include hypotheses on positivity. To understand why, it is important to observe that in \cite[Definition 2.8]{ouhabaz2005analysis} 
irreducibility is defined as follows. A holomorphic semigroup $T$ (in particular, a semigroup associated with an elliptic form) 
on $L^2(\Omega,\mu)$ is irreducible if and only if
\begin{equation}\label{eq:irreducibility}
T(t)f>0 \ \text{a.e. on}\ \Omega \ (t>0)\ \text{whenever}\ 0\neq f \in L^2(\Omega,\mu)_+.
\end{equation}
This defining property is easily seen to imply the invariance property we have used to define irreducibility and already implies that, in 
particular, an irreducible semigroup is positive. On the other hand, for positive semigroups, both concepts of irreducibility are equivalent; 
this is the content of \cite[Theorem 2.9]{ouhabaz2005analysis}. To summarize, our definition of 
irreducibility here dispenses with hypotheses on positivity because these hypotheses are usually required only to go from the 
invariance property we have used to define irreducibility to the property expressed in Eq. \eqref{eq:irreducibility}. Moreover, locality of the 
form is also not needed to arrive at the conclusion in the hard part of the proof above (although the forms to which we apply the result are local); this 
can in part be explained since the form domain is very special, namely, a subspace of $H^1(\Omega)$ and $\Omega$ is connected. 
\end{remark}

Finally, we will also need some monotonicity properties of the semigroups $e^{-tA_\beta}$ when different $\beta$'s are considered. Let 
$\Omega\subset \mathbb{R}^N$ be a bounded open set with Lipschitz boundary. Suppose that the form $\frak{a}$ in Eq. \eqref{eq:associated_form} has 
real coefficients. By \cite[Theorem 4.2]{ouhabaz2005analysis} the semigroup generated by $-A_\beta$ is positive. Moreover, it follows 
from well-known comparison results, see e.g. \cite[Theorem 2.24]{ouhabaz2005analysis}, that if $\beta_0,\beta_1\in L^\infty(\partial\Omega)$ 
and $0<\beta_0\leqslant \beta_1$ then $e^{-tA_{\beta_1}}\leqslant e^{-tA_{\beta_0}}$ for all $t\geqslant 0$.  If $\Omega$ is also connected then 
the semigroup generated by $-A_\beta$ is irreducible, by Proposition \ref{proposition:irreducibility}. In this case, the Krein-Rutman theorem 
(more precisely, the monotonicity result \cite[Theorem 10.2.10]{arendt:9ISEM}) allows us to summarize this discussion in the following.

\begin{proposition}\label{proposition:first_eigenvalue-characterize-operators}
Let $\Omega\subset \mathbb{R}^N$ be a bounded open connected set with Lipschitz boundary. Suppose the form 
$\frak{a}$ in Eq. \eqref{eq:associated_form} has real coefficients and let $\beta_0,\beta_1\in L^\infty(\partial\Omega)$ with 
$0<\beta_0\leqslant \beta_1$. Then 
\begin{equation}\label{eq:first_eingenvalue-iff-operators}
\lambda_1(A_{\beta_0})=\lambda_1(A_{\beta_1})\ \text{if, and only if}\ A_{\beta_0}=A_{\beta_1}.
\end{equation} 
\end{proposition}

\section{The Dirichlet-to-Neumann operator $D_\lambda^\mathscr{A}$}\label{sec:DtN-operator}

Now we turn to the definition of the Dirichlet-to-Neumann operator with respect to Problem \eqref{eq:eigenvalue-problem-DtN-perturb}. Under 
suitable assumptions on $\Omega$ and $\mathscr{A}$ (see Proposition \ref{prop:operator-DtN} below) we define 
the Dirichlet-to-Neumann operator $D_\lambda^\mathscr{A}$ as the operator on $L^2(\partial\Omega)$ such that 
$\varphi\in \mathscr{D}(D_\lambda^\mathscr{A})$ and $D_\lambda^\mathscr{A}\varphi =h$ if, and only if there is a weak solution $u\in H^1(\Omega)$  of 
Problem \eqref{eq:eigenvalue-problem-DtN-perturb} with $\partial_\nu u=h$ (weak conormal derivative). As we have anounced in the introduction 
(see paragraph before Eq. \eqref{eq:form-dtn-perturb}), 
let us see that $D_\lambda^\mathscr{A}$ is the operator associated to the form $\mathfrak{a}_\lambda$ defined in Eq. \eqref{eq:form-dtn-perturb} when 
$j:H^1(\Omega)\to L^2(\partial\Omega)$ is the trace operator. The form $\mathfrak{a}_\lambda$ is not $j$-elliptic in general 
and the fact that it admits a well-defined asociated operator with good properties can be established in a reasonably easy way either 
by appealing to the theory of compactly elliptic forms (which we recall at the end of this section) or by using the following result. 

For a bounded sequilinear form $\mathfrak{a}:V\times V\to \mathbb{K}$, let $V_j(\mathfrak{a})$ be the closed subspace 
\begin{equation}\label{eq:Vj}
V_j(\mathfrak{a}):=\{u\in V:\mathfrak{a}(u,v)=0\ \text{for all}\ v\in \mathscr{N}(j)\}.
\end{equation}

\begin{proposition}[cf. \cite{arendt2012sectorial}, Corollary 2.2]\label{prop:associated_operator-V-kerj-Vj}
Let $V$, $H$ be Hilbert spaces and let $j\in \mathscr{L}(V,H)$ have dense range. Let $\mathfrak{a}:V\times V\to \mathbb{K}$ be a continuous 
sesquilinear form and suppose that
\begin{enumerate}
\item[(i)] $V=V_j(\mathfrak{a})+\mathscr{N}(j)$;
\item[(ii)] there exists $\omega\in \mathbb{R}$ and $\alpha>0$ such that ${\rm Re}\, \mathfrak{a}(u)+\omega\|j(u)\|_H^2\geqslant \alpha\|u\|_V^2$ 
for all $u\in V_j(\mathfrak{a})$.
\end{enumerate}
Then $\mathfrak{a}$ admits an associated operator $A$ in the sense of \eqref{eq:associated_operator} which is m-sectorial.
\end{proposition}

The following result describes the Dirichlet-to-Neumann operator $D_\lambda^\mathscr{A}$ along with its basic properties.

\begin{proposition}\label{prop:operator-DtN}
Let $\Omega\subset \mathbb{R}^N$ be a bounded open set with Lipschitz boundary and assume the conditions stated in 
Hypothesis \ref{hypothesis}. Suppose $\lambda\notin \sigma(A_D)$. Then the operator $D_\lambda^\mathscr{A}$ 
on $L^2(\partial\Omega)$ given by 
\begin{align*}
\mathscr{D}(D_\lambda^\mathscr{A})&=\{\varphi\in L^2(\partial\Omega):\text{there exists}\ u\in H^1(\Omega)\ \text{such that}\ \\
&\hspace{5cm}  \mathscr{A}u=\lambda u, \ u|_{\partial\Omega}=\varphi\ \text{and}\ \partial_{\nu}u\in L^2(\partial\Omega)\}\\
D_\lambda^\mathscr{A}\varphi &=\partial_{\nu}u ,
\end{align*}
is quasi-m-accretive and has compact resolvent. If $b=\overline{c}$ and $d$ is real then $D_\lambda^\mathscr{A}$ is self-adjoint. 
If $\mathbb{K}=\mathbb{C}$ then $D_\lambda^\mathscr{A}$ is quasi-m-sectorial. Moreover, $D_\lambda^\mathscr{A}$ is the operator associated 
with the restrictions of $\mathfrak{a}_\lambda$ and $j$ to 
\begin{equation}\label{eq:V-j-a_lambda}
V_j(\mathfrak{a}_\lambda)=\{u\in H^1(\Omega):\mathscr{A}u=\lambda u\ \text{in the distributional sense}\}. 
\end{equation}
\end{proposition}

\begin{proof}
Let us first show that (i) in Proposition \ref{prop:associated_operator-V-kerj-Vj} is satisfied. To this end, we follow the same reasoning 
used in \cite[Lemma 2.2]{arendt2012friedlander} which concerns 
the case $a=I$, $b=c=0$ and $d=0$. With our notation, we have to prove that $V=H_0^1(\Omega)\oplus V_j(\mathfrak{a}_\lambda)$ 
if $\lambda\notin \sigma(A_\text{D})$. Define $\mathscr{L}:H_0^1(\Omega)\to H^{-1}(\Omega)$ by 
$\langle\mathscr{L}u,v\rangle =\mathfrak{a}(u,v)$. Under the usual identifications in the Gel'fand triple 
$H_0^1(\Omega)\hookrightarrow L^2(\Omega)\hookrightarrow H^{-1}(\Omega)$ we see that 
$A_\text{D}$ is the part of $\mathscr{L}$ on $L^2(\Omega)$ so that $\sigma(\mathscr{L})=\sigma(A_\text{D})$ by 
\cite[Proposition 3.10.3]{Arendtetall:2011}. Thus $\lambda -\mathscr{L}$ is invertible. Fix $u\in H^1(\Omega)$ and define $\Phi\in H^{-1}(\Omega)$ 
by $\Phi(v)=\mathfrak{a}_\lambda(u,v)$. There exists $u_0\in H_0^1(\Omega)$ such that $\lambda u_0 -\mathscr{L}u_0 =-\Phi$, that is,
\[
\lambda \int_\Omega u_0\overline{v}\, dx-\mathfrak{a}(u_0,v)=-\mathfrak{a}_\lambda(u,v)\ \ \ (v\in H^1_0(\Omega)),
\]
thus $\mathfrak{a}_\lambda(u-u_0,v)=0$ for all $v\in H^1_0(\Omega)$. Therefore $u-u_0\in V_j(\mathfrak{a}_\lambda)$ and it follows 
that $H^1(\Omega)=H_0^1(\Omega)+V_j(\mathfrak{a}_\lambda)$. Since $\lambda\notin \sigma(A_\text{D})$ we also 
have $H_0^1(\Omega)\cap V_j(\mathfrak{a}_\lambda)=\{0\}$. 

Now let us prove the ellipticity of $\mathfrak{a}_\lambda|_{V_j(\mathfrak{a}_\lambda)}$, that is, condition (ii) in Proposition 
\ref{prop:associated_operator-V-kerj-Vj}. Lions's lemma applied to the (compact) immersion $V_j(\mathfrak{a}_\lambda)\hookrightarrow L^2(\Omega)$ and 
to the (injective) restriction of the trace $j:V_j(\mathfrak{a}_\lambda)\rightarrow L^2(\partial\Omega)$ gives, for each $\delta>0$, a constant 
$c_\delta\geqslant 0$ such that 
\begin{equation}\label{eq:estimate-Lions-L2-H1-trace}
\int_\Omega|u|^2\, dx\leqslant \delta \|u\|_{H^1(\Omega)}^2+c_\delta\int_{\partial\Omega}|u|^2\, d\sigma , 
\end{equation}
for all $u\in V_j(\mathfrak{a}_\lambda)$, so that 
$$
\int_\Omega|u|^2\, dx\leqslant 
\frac{\delta}{1-\delta}\int_\Omega|\nabla u|^2\, dx +\frac{c_\delta}{1-\delta}\int_{\partial\Omega}|u|^2\, d\sigma \ \ \ (u\in V_j(\mathfrak{a}_\lambda)).
$$
From estimate \eqref{eq:estimate-form-elliptic-H10} (see Remark \ref{remark:basic_estimate}) we have 
\[
{\rm Re}\, \frak{a}_\lambda(u)+(|\lambda|+\omega_1) \|u\|_{L^2(\Omega)}^2 \geqslant \frac{\kappa}{2}\|u\|_{H^1(\Omega)}^2\ \ \ (u\in H^1(\Omega)). 
\]
Combining the above two estimates we obtain
\[
{\rm Re}\, \frak{a}_\lambda(u)+\frac{(|\lambda|+\omega_1)\delta}{1-\delta} \int_\Omega|\nabla u|^2\, dx 
+\frac{(|\lambda|+\omega_1)c_\delta}{1-\delta}\int_{\partial\Omega}|u|^2\, d\sigma
\geqslant \frac{\kappa}{2}\|u\|_{H^1(\Omega)}^2\ \ \ (u\in V_j(\mathfrak{a}_\lambda)). 
\]
By choosing $\delta>0$ satisfying $\frac{(|\lambda|+\omega_1)\delta}{1-\delta}=\frac{\kappa}{4}$ and then $\omega$ to be the correponding number 
$\frac{(|\lambda|+\omega_1)c_\delta}{1-\delta}$ we arrive at the estimate
\begin{equation}\label{eq:ellipticity-a-lambda}
{\rm Re}\, \frak{a}_\lambda(u)+\omega\|u|_{\partial\Omega}\|^2_{L^2(\partial\Omega)}
\geqslant \frac{\kappa}{4}\|u\|_{H^1(\Omega)}^2\ \ \ (u\in V_j(\mathfrak{a}_\lambda)).  
\end{equation}
This is precisely the ellipticity of $\mathfrak{a}_\lambda|_{V_j(\mathfrak{a}_\lambda)}$ with respect to the trace. 

The description of $D_\lambda^\mathscr{A}$ follows from standard arguments, based on the definition 
of weak conormal derivative in Eq. \eqref{eq:derivada_conormal_fraca}; see e.g. \cite[Lecture 8]{arendt-voigt:18ISEM}.
\end{proof}

\begin{proposition}\label{prop:spectrum-Robin-DtN}
Under the same assumptions as in Proposition \ref{prop:operator-DtN}, let $\lambda\notin \sigma(A_\text{D})$ 
and $\beta\in \mathbb{K}$.
\begin{enumerate}
\item $\mathscr{N}(\beta +D_\lambda^\mathscr{A})$ and $\mathscr{N}(\lambda -A_\beta)$ are isomorphic.
\item $-\beta\in \sigma (D_\lambda^\mathscr{A})$ if, and only if $\lambda\in \sigma(A_\beta)$.
\end{enumerate}
\end{proposition}

\begin{proof}
(a). We prove that $\mathscr{N}(\lambda -A_\beta)\ni u\mapsto u|_{\partial\Omega}\in \mathscr{N}(\beta +D_\lambda^\mathscr{A})$ is an 
isomorphism. 

First, we check this operator is well-defined. Let $\mathfrak{a}^\beta$ be the form defined by 
Eq. \eqref{eq:forma-Robin-beta}. Thus, $u\in \mathscr{N}(\lambda -A_\beta)$ if, and only if $\mathfrak{a}^\beta(u,v)=\lambda(u|v)_{L^2(\Omega)}$ 
for all $v\in H^1(\Omega)$, or 
\begin{equation}\label{eq1:dem:teorema:espectro-Robin-DtN}
\mathfrak{a}(u,v)+\beta \int_{\partial\Omega} u\overline{v}\, d\sigma=\lambda\int_\Omega u\overline{v}\, dx\ \ \ (v\in H^1(\Omega)).
\end{equation}
Equivalently, $\mathfrak{a}_\lambda(u,v)=(-\beta u|v)_{L^2(\partial\Omega)}$ for all $v\in H^1(\Omega)$, 
which means that $u|_{\partial\Omega}\in \mathscr{D}(D_\lambda^\mathscr{A})$ and 
$D_\lambda^\mathscr{A} u|_{\partial\Omega}=-\beta u|_{\partial\Omega}$, 
that is, $u|_{\partial\Omega}\in \mathscr{N}(\beta +D_\lambda^\mathscr{A})$.

Now, we prove the above operator is surjective. If $\varphi\in \mathscr{N}(\beta +D_\lambda^\mathscr{A})$ then 
$D_\lambda^\mathscr{A}\varphi=-\beta\varphi$ and this means that, for some $u\in H^1(\Omega)$, 
$u|_{\partial\Omega}=\varphi$ and $\mathfrak{a}_\lambda(u,v)=(-\beta \varphi|v)_{L^2(\partial\Omega)}$ for all $v\in H^1(\Omega)$. This is 
clearly equivalent to Eq. \eqref{eq1:dem:teorema:espectro-Robin-DtN}; thus, $\varphi=u|_{\partial\Omega}$ for some 
$u\in \mathscr{N}(\lambda -A_\beta)$.

Finally, we show injectivity. If $u\in \mathscr{N}(\lambda -A_\beta)$ and $u|_{\partial\Omega} =0$ then, by 
\eqref{eq1:dem:teorema:espectro-Robin-DtN}, $\mathfrak{a}_\lambda(u,v)=0$ for all $v\in H^1(\Omega)$; in particular, 
$\mathscr{A}u=\lambda u$ (in the distributional sense). Thus, $A_\text{D}u=\lambda u$ and, since $\lambda\notin \sigma(A_\text{D})$, 
we conclude that $u=0$. 

(b). Follows from (a) and the fact that $\sigma (D_\lambda^\mathscr{A})$ (resp. $\sigma(A_\beta)$) is a pure point spectrum, since 
$D_\lambda^\mathscr{A}$ (resp. $A_\beta$) has compact resolvent.
\end{proof}

A sequilinear form $\mathfrak{a}:V\times V\to \mathbb{K}$ is said to be \textit{compactly elliptic} if there exists a Hilbert 
space $\widetilde{H}$ and a compact operator $\widetilde{j}:V\to \widetilde{H}$ such that $\mathfrak{a}$ is `$\widetilde{j}$-elliptic' 
in the sense that, for some constant $\widetilde{\alpha}>0$,
$$
{\rm Re}\, \mathfrak{a}(u)+\|\widetilde{j}(u)\|_{\widetilde{H}}^2\geqslant \widetilde{\alpha}\|u\|_V^2\ \ \ (u\in V).
$$
This notion has been introduced in \cite{arendt2014dirichlet}. Now, consider the condition
\begin{equation}\label{eq:condition-form}
\text{if}\ u\in V, \ j(u)=0\ \text{and}\ \mathfrak{a}(u,v)=0\ \text{for all}\ v\in \mathscr{N}(j),\ \text{then}\ u=0.
\end{equation}
If $\mathfrak{a}$ is compactly elliptic and satisfies \eqref{eq:condition-form} then it follows from Lions's lemma that condition (ii) in 
Proposition \ref{prop:associated_operator-V-kerj-Vj} is satisfied. In fact, \eqref{eq:condition-form} implies that the restriction 
$j:V_j(\frak{a})\rightarrow H$ is injective and the operador $\widetilde{j}:V_j(\frak{a})\rightarrow \widetilde{H}$ is compact by 
hypothesis. From Lions's lemma, 
$$\|\widetilde{j}(u)\|^2_{\widetilde{H}}\leqslant \frac{\widetilde{\alpha}}{2}\|u\|_V^2+c\|j(u)\|_H^2\ \ \ (u\in V_j(\frak{a}))$$
for some constant $c\geqslant 0$. The aforementioned condition (ii) follows with $\alpha=\frac{\widetilde{\alpha}}{2}$ by combining the above 
with the compactly ellipticity estimate. Note that this argument is nothing more that an abstract counterpart of the reasoning leading to 
estimate \eqref{eq:estimate-Lions-L2-H1-trace}. It is easy to see that, under condition \eqref{eq:condition-form}, a well-defined operator can be 
associated to $\mathfrak{a}$ (compact ellipticity is not required for this) through \eqref{eq:associated_operator}. The following result tell 
us that much more can be derived from this construction.

\begin{proposition}[cf. \cite{arendt-voigt:18ISEM}, Proposition 8.10 and Theorem 8.11]\label{prop:associated_operator-compactly_elliptic_form}
Let $V$, $H$ be Hilbert spaces and let $j\in \mathscr{L}(V,H)$ have dense range. Let $\mathfrak{a}:V\times V\to \mathbb{K}$ be a continuous 
sesquilinear form. If $\mathfrak{a}$ is compactly elliptic and satisfies \eqref{eq:condition-form} then conditions (i) and (ii) in 
Proposition \ref{prop:associated_operator-V-kerj-Vj} hold. Moreover, the associated operator on $H$ concides with the operator associated to the 
elliptic form obtained from $\mathfrak{a}$ and $j$ by restriction to $V_j(\mathfrak{a})$.
\end{proposition}

Let us see that $\mathfrak{a}_\lambda$ is compactly elliptic and satisfies \eqref{eq:condition-form} with $V=H^1(\Omega)$ and 
$j$ being the trace from $H^1(\Omega)$ to $L^2(\partial\Omega)$. On the one hand, from the estimate in the 
proof of Proposition \ref{prop:operator-delta-plus-nabla-H10} (see also Remark \ref{remark:basic_estimate}) we can infer that
$$
{\rm Re}\, \frak{a}_\lambda(u)+(|\lambda|+\omega_1) \|u\|_{L^2(\Omega)}^2 \geqslant \frac{\kappa}{2}\|u\|_{H^1(\Omega)}^2\ \ \ (u\in H^1(\Omega)),
$$
which implies that $\mathfrak{a}_\lambda$ is compactly elliptic with $\widetilde{H}=L^2(\Omega)$ and $\widetilde{j}$ being the 
multiplication of the embedding $H^1(\Omega)\hookrightarrow L^2(\Omega)$ by 
$\sqrt{|\lambda| +\omega_1}$. On the other hand, condition \eqref{eq:condition-form} means here that \emph{if $u\in H^1(\Omega)$, 
$u|_{\partial\Omega}=0$, i.e. $u\in H_0^1(\Omega)$, and $\mathfrak{a}_\lambda(u,v)=0$ for all $v\in H^1_0(\Omega)$, then $u=0$}. 
Accordingly, let $u\in H_0^1(\Omega)$ and suppose $\mathfrak{a}_\lambda(u,v)=0$ for all 
$v\in H_0^1(\Omega)$. Thus $u\in H_0^1(\Omega)$ and $\mathfrak{a}(u,v)=\lambda\int_\Omega u\overline{v}\, dx$ for all 
$v\in C_\text{c}^\infty(\Omega)$, which means that $A_\text{D}u=\lambda u$. Therefore, if 
$\lambda\notin \sigma(A_\text{D})$ then $u=0$, that is, $\mathfrak{a}_\lambda$ satisfies \eqref{eq:condition-form}. In view of 
Proposition \ref{prop:associated_operator-compactly_elliptic_form}, the above arguments give another proof of 
Proposition \ref{prop:operator-DtN}.

\section{Proof of Theorem \ref{thm:main}}\label{sec:DtN-semigroup}

In this section we prove Theorem \ref{thm:main}. Let $\Omega\subset \mathbb{R}^N$ be a bounded open set with 
Lipschitz boundary and assume Hypothesis \ref{hypothesis}. Let us define the number
\begin{equation}
\begin{aligned}
\lambda_1^\text{D}(\mathfrak{a})&:=\inf_{u\in H_0^1(\Omega),u\neq 0}\frac{{\rm Re}\, \mathfrak{a}(u)}{\|u\|_{L^2(\Omega)}^2}
\label{eq:definition-lambda1-AD}\\
&=\inf_{u\in H_0^1(\Omega),u\neq 0}
\frac{\int_\Omega a\nabla u\cdot \overline{\nabla u}\, dx 
+{\rm Re}\,\int_\Omega ((b+\overline{c})\cdot\nabla u)\overline{u}\, dx+\int_\Omega ({\rm Re}\, d)|u|^2\, dx}{\int_{\Omega}|u|^2\, dx}. 
\end{aligned} 
\end{equation}
The number $\lambda_1^\text{D}(\mathfrak{a})$ coincides with the first eigenvalue $\lambda_1^\text{D}$ of the Dirichlet Laplacian when 
$\mathscr{A}=-\Delta$ (that is, $a=I$, $b=c=0$ and $d=0$). From the variational characterization of $\lambda_1^\text{D}$ in 
Eq. \eqref{eq:first_eigenvalue-D_Laplacian} it follows from the estimate \eqref{eq:estimate-re-intb-gradu-u-2} in the 
proof of Proposition \ref{prop:operator-delta-plus-nabla-H10} that 
\begin{align}
{\rm Re}\, \mathfrak{a}(u)&\geqslant \kappa\int_\Omega|\nabla u|^2\, dx
-\|b+\overline{c}\|_{L^\infty(\Omega)^N}\|\nabla u\|_{L^2(\Omega)^N} \|u\|_{L^2(\Omega)}-\|({\rm Re}\, d)^-\|_{L^\infty}\int_{\Omega}|u|^2\, dx\\
&\geqslant \kappa\int_\Omega|\nabla u|^2\, dx
-\frac{\|b+\overline{c}\|_{L^\infty(\Omega)^N}}{\sqrt{\lambda_1^\text{D}}}
\|\nabla u\|_{L^2(\Omega)^N}^2-\|({\rm Re}\, d)^-\|_{L^\infty}\int_{\Omega}|u|^2\, dx\\
&\geqslant \Big(\kappa -\frac{\|b+\overline{c}\|_{L^\infty(\Omega)^N}}{\sqrt{\lambda_1^\text{D}}}\Big)\lambda_1^\text{D}\int_\Omega|u|^2\, dx
-\|({\rm Re}\, d)^-\|_{L^\infty(\Omega)}\int_{\Omega}|u|^2\, dx ,
\end{align}
provided $\|b+\overline{c}\|_{L^\infty(\Omega)^N}<\kappa \sqrt{\lambda_1^\text{D}}$; in this case
\begin{equation}\label{eq:estimate-first-eigenvalue-AD-DeltaD}
\lambda_1^\text{D}(\mathfrak{a})\geqslant \kappa\lambda_1^\text{D} -\|b+\overline{c}\|_{L^\infty(\Omega)^N}\sqrt{\lambda_1^\text{D}}
-\|({\rm Re}\, d)^-\|_{L^\infty(\Omega)}. 
\end{equation}
In the following remarks we show how a more explicit estimate can be obtained under additional assumptions on $b$, $c$ and $d$. 

\begin{remark}\label{rmk:int-b-gradu-u-positivo}
The following is well known to the experts and is included here for the
convenience of the reader. Suppose, in addition to the standing assumptions on $\Omega$, $a$ and $d$, 
that $b,c\in C^1(\overline{\Omega})^N$ are real vector fields. 
\begin{enumerate}
\item \label{rmk:item:b-nabla-skew} If $u,v\in H^1(\Omega)$ then 
\begin{align*}
\int_{\Omega}(b\cdot \nabla u)\overline{v}\, dx&=\sum_{j=1}^N\int_\Omega b_j\partial_ju\overline{v}\, dx\\
&=\sum_{j=1}^N\int_\Omega \partial_j(b_ju\overline{v})\, dx-\sum_{j=1}^N\int_\Omega u\partial_j(b_j\overline{v})\, dx\\
&=\int_{\partial\Omega} u\overline{v}b\cdot \nu\, d\sigma-\int_\Omega (\text{div}\, b)u\overline{v}\, dx-\int_{\Omega}u(b\cdot \overline{\nabla v}).
\end{align*}
Therefore, if $\text{div}\, b=0$ and $b\cdot \nu =0$ then $b\cdot \nabla $ is skew-Hermitian in the sense that 
\begin{equation}\label{eq:b-nabla-skew}
\int_\Omega (b\cdot \nabla u)\overline{v}\, dx=- \int_\Omega (b\cdot \overline{\nabla v})u\, dx\ \ \ (u,v\in H^1(\Omega)) .
\end{equation} 
In particular, ${\rm Re}\,\int_\Omega (b\cdot \nabla u)\overline{u}\, dx= 0$ for all $u\in H^1(\Omega)$.
\item \label{rmk:item:estimate-lambda1D} It follows from the computation in \ref{rmk:item:b-nabla-skew} above 
(with $b$ replaced by $b+c$) that
\begin{equation}
\begin{aligned}
\lambda_1^\text{D}(\mathfrak{a})&\geqslant\inf_{u\in H_0^1(\Omega),u\neq 0}
\frac{\int_\Omega a\nabla u\cdot \overline{\nabla u}\, dx +\int_\Omega ({\rm Re}\, d)|u|^2\, dx}{\|u\|_{L^2(\Omega)}^2}\\
&\hspace{1cm}+\frac{1}{2}\inf_{u\in H_0^1(\Omega),u\neq 0}
\frac{\int_{\partial\Omega}|u|^2(b+c)\cdot \nu \, d\sigma-\int_\Omega {\rm div}\, (b+c)|u|^2\, dx}{\|u\|_{L^2(\Omega)}^2}.  
\end{aligned}
\end{equation}
Thus, under the hypotheses in Theorem \ref{thm:main}, 
$\lambda_1^\text{D}(\mathfrak{a})\geqslant \kappa\lambda_1^\text{D}-\|({\rm Re}\, d)^-\|_{L^\infty(\Omega)}$. In particular, 
if ${\rm Re}\, d\geqslant 0$ then $\lambda_1^\text{D}(\mathfrak{a})\geqslant \kappa\lambda_1^\text{D}$.
\end{enumerate}
\end{remark}

The rest of this paper is devoted to the proof of Theorem \ref{thm:main}. For simplicity, we restrict to real scalars, 
so that, in particular, 
\[
\int_\Omega (b\cdot \nabla u)v\, dx=-\int_\Omega (b\cdot \nabla v)u\, dx\ \ \ (u,v\in H^1(\Omega)) .
\]
It may be interesting to carry out all the details concerning the case of complex scalars; we leave this task to the interested reader.
Note that, in view of Remark \ref{rmk:int-b-gradu-u-positivo}\ref{rmk:item:estimate-lambda1D}, the condition 
$4\kappa^{-1}\|b-c\|^2_{L^\infty(\Omega)^N}+\|d^-\|_{L^\infty(\Omega)}+\lambda <\kappa\lambda_1^\text{D}$ stated in 
Theorem \ref{thm:main}\ref{item:semigroup-DtN-positive}\ref{item:semigroup-DtN-submarkovian} 
implies the more `abstract' condition $4\kappa^{-1}\|b-c\|^2_{L^\infty(\Omega)^N}+\lambda <\lambda_1^\text{D}(\mathfrak{a})$, which is 
actually what we use in the proofs below. The same applies to the conditions in 
Theorem \ref{thm:main}\ref{item:semigroup-DtN-irreducible}\ref{item:semigroup-DtN-domination}. 

\begin{proof}[Proof of Theorem \ref{thm:main}\ref{item:semigroup-DtN-positive}]
We apply Ouhabaz invariance criterion to the closed convex set $C:=\{u\in L^2(\Omega;\mathbb{R}):u\geqslant 0\}$, 
whose minimizing projection is $Pu=u^+$.

From the lattice properties of $H^1(\Omega)$ and properties of the trace we know that if $u\in H^1(\Omega)$ then $u^+\in H^1(\Omega)$ and
$(u|_{\partial\Omega})^+=u^+|_{\partial\Omega}$. A similar statement holds for $u^-$. On the other hand, there exists 
$u_0, \widetilde{u_0}\in H_0^1(\Omega)$ and $u_1, u_2\in V_j(\mathfrak{a}_\lambda)$ such that 
$$
u^+=u_0+u_1 \ \ \ \text{and}\ \ \ u^-=\widetilde{u_0}+u_2.
$$
But $u=u^+ -u^-=(u_0-\widetilde{u_0})+(u_1-u_2)$, thus $u_0=\widetilde{u_0}$ if $u\in V_j(\mathfrak{a}_\lambda)$; we can assume that 
$u\in V_j(\mathfrak{a}_\lambda)$ from the start, since our Dirichlet-to-Neumann operator is also associated to the restrictions of the trace 
and the form $\mathfrak{a}_\lambda$ to $V_j(\mathfrak{a}_\lambda)$ (see Proposition \ref{prop:associated_operator-compactly_elliptic_form}). Therefore 
$(u|_{\partial\Omega})^+=u_1|_{\partial\Omega}$ and $(u|_{\partial\Omega})^-=u_2|_{\partial\Omega}$; this means that $Pj(u)=j(u_1)$ 
and it suffices to show that $\mathfrak{a}_\lambda(u,u-u_1)\geqslant -\omega\|j(u-u_1)\|_{L^2(\partial\Omega)}^2$ for some $\omega\in \mathbb{R}$ 
(depending only on $\mathfrak{a}_\lambda$); since $u-u_1=-u_2$, this amounts to show that 
$\mathfrak{a}_\lambda(u,u_2)\leqslant \omega\|j(u_2)\|_{L^2(\partial\Omega)}^2$ or, equivalently,
\begin{equation}
\mathfrak{a}_\lambda(u_1,u_2)\leqslant \mathfrak{a}_\lambda(u_2)+\omega\int_{\partial\Omega}|u_2|^2\, d\sigma .
\end{equation}
From estimate \eqref{eq:ellipticity-a-lambda}, it is enough to show that 
$\mathfrak{a}_\lambda(u_1,u_2)\leqslant \frac{\kappa}{4}\|u_2\|_{H^1(\Omega)}^2$. Let us show that this can always be achieved under 
the hypotheses in Theorem \ref{thm:main}(a).

First, note that 
\begin{equation}
\mathfrak{a}_\lambda(u_1,u_2)=\mathfrak{a}_\lambda(u^+,u^-)-\mathfrak{a}_\lambda(u_0,u_2)-\mathfrak{a}_\lambda(u_1,u_0)
-\mathfrak{a}_\lambda(u_0,u_0).
\end{equation}
Clearly,
\begin{align*}
&\mathfrak{a}_\lambda(u^+,u^-)\\
&=\int_\Omega a\nabla u^+\cdot \nabla u^-\, dx +\int_\Omega (b\cdot\nabla u^+)u^-\, dx+\int_\Omega u^+(c\cdot\nabla u^-)\, dx
+\int_\Omega du^+ u^-\, dx-\lambda \int_\Omega u^+ u^-\, dx\\
&=0.
\end{align*}
Moreover, since $u_1\in V_j(\mathfrak{a}_\lambda)$ and $u_0\in H_0^1(\Omega)$, we have $\mathfrak{a}_\lambda(u_1,u_0)=0$. Besides, 
\begin{align*}
&-\mathfrak{a}_\lambda(u_0,u_2)\\
&=-\int_\Omega a\nabla u_0\cdot \nabla u_2\, dx-\int_\Omega (b\cdot\nabla u_0)u_2\, dx-\int_\Omega u_0(c\cdot\nabla u_2)\, dx
-\int_\Omega du_0 u_2\, dx+\lambda \int_\Omega u_0 u_2\, dx\\
&\stackrel{(1)}{=}-\int_\Omega a\nabla u_2\cdot \nabla u_0\, dx+\int_\Omega (b\cdot\nabla u_2)u_0\, dx+\int_\Omega u_2(c\cdot\nabla u_0)\, dx
-\int_\Omega du_2 u_0\, dx+\lambda \int_\Omega u_2 u_0\, dx\\
&\stackrel{(2)}{=}2\Big(\int_\Omega (b\cdot\nabla u_2)u_0\, dx+\int_\Omega u_2(c\cdot\nabla u_0)\, dx\Big)\\
&=2\int_\Omega ((b-c)\cdot\nabla u_2)u_0\, dx\\
&\leqslant 2\|b-c\|_{L^\infty(\Omega)}\|\nabla u_2\|_{L^2(\Omega)^N}\|u_0\|_{L^2(\Omega)}\\
&\leqslant \frac{\kappa}{4}\|\nabla u_2\|_{L^2(\Omega)^N}^2+\frac{4\|b-c\|_{L^\infty(\Omega)}^2}{\kappa}\|u_0\|_{L^2(\Omega)}^2.
\end{align*}
Above, identity (1) follows from Remark \ref{rmk:int-b-gradu-u-positivo}\ref{rmk:item:b-nabla-skew}, which asserts that 
$$
\int_\Omega (b\cdot\nabla u_0)u_2\, dx=-\int_\Omega (b\cdot\nabla u_2)u_0\, dx
$$
with an analogous identity for $c$. Identity (2) follows from the fact that $\mathfrak{a}_\lambda(u_2,u_0)=0$ 
(since $u_2\in V_j(\mathfrak{a}_\lambda)$ and $u_0\in H_0^1(\Omega)$) and the estimates are the usual H\"older and Young's inequalities, 
respectively. Thus,
\begin{align*}
\mathfrak{a}_\lambda(u_1,u_2)&\leqslant \frac{\kappa}{4}\|\nabla u_2\|_{L^2(\Omega)^N}^2\\
& \hspace{1cm} -\int_\Omega a\nabla u_0\cdot \nabla u_0\, dx-\int_\Omega (b\cdot\nabla u_0)u_0\, dx-\int_\Omega u_0(c\cdot\nabla u_0)\, dx
-\int_\Omega du_0^2\, dx\\
&\hspace{2cm} +\big(\lambda+4\kappa^{-1}\|b-c\|_{L^\infty(\Omega)}^2\big)\int_\Omega |u_0|^2\, dx\\
&\leqslant \frac{\kappa}{4}\|u_2\|_{H^1(\Omega)}^2\\
& \hspace{1cm}-\int_\Omega a\nabla u_0\cdot \nabla u_0\, dx-\int_\Omega (b\cdot\nabla u_0)u_0\, dx-\int_\Omega u_0(c\cdot\nabla u_0)\, dx
-\int_\Omega du_0^2\, dx\\
&\hspace{2cm} +\lambda_1^\text{D}(\mathfrak{a})\int_\Omega |u_0|^2\, dx\\
& \leqslant \frac{\kappa}{4}\|u_2\|_{H^1(\Omega)}^2,
\end{align*}
where the last estimate above follows from the definition of $\lambda_1^\text{D}(\mathfrak{a})$ in \eqref{eq:definition-lambda1-AD}.
\end{proof}

Now, we turn to the proof of Theorem \ref{thm:main}\ref{item:semigroup-DtN-submarkovian}. As it is well known, the sub-Markovian property is 
equivalent to the invariance of the closed convex set $C=\{u\in L^2(\Omega):u\leqslant 1\}$. We then apply Ouhabaz's invariance criterion with 
$Pu=u\wedge 1$.

\begin{proof}[Proof of Theorem \ref{thm:main}\ref{item:semigroup-DtN-submarkovian}]
From the lattice properties of $H^1(\Omega)$ and properties of the trace we know that if $u\in H^1(\Omega)$ then $u\wedge 1\in H^1(\Omega)$ and 
$(u\wedge 1)|_{\partial\Omega}=u|_{\partial\Omega}\wedge 1$. On the other hand, there exists 
$u_0, \widetilde{u_0}\in H_0^1(\Omega)$ and $u_1, u_2\in V_j(\mathfrak{a}_\lambda)$ such that 
$$
u\wedge 1=u_0+u_1 \ \ \ \text{and}\ \ \ -(u-1)^+=\widetilde{u_0}+u_2.
$$
But $u=u\wedge 1+(u-1)^+=(u_0-\widetilde{u_0})+(u_1-u_2)$, thus $u_0=\widetilde{u_0}$ if $u\in V_j(\mathfrak{a}_\lambda)$. Therefore 
$(u|_{\partial\Omega})\wedge 1=u_1|_{\partial\Omega}$ (and $(u|_{\partial\Omega}-1)^+=-u_2|_{\partial\Omega}$), so that $Pj(u)=j(u_1)$,  
$u-u_1=-u_2$ and, as before, we must show that $\mathfrak{a}_\lambda(u_1,u_2)\leqslant \frac{\kappa}{4}\|u_2\|_{H^1(\Omega)}^2$. Using the 
estimates in the proof of (a) above we find that 
\begin{align*}
&\mathfrak{a}_\lambda(u_1,u_2)\\
&=\mathfrak{a}_\lambda(u\wedge 1,-(u-1)^+)-\mathfrak{a}_\lambda(u_0,u_2)-\mathfrak{a}_\lambda(u_1,u_0)-\mathfrak{a}_\lambda(u_0,u_0)\\
&\leqslant -\int_\Omega a\nabla (u\wedge 1)\cdot \nabla (u-1)^+\, dx -\int_\Omega [b\cdot\nabla (u\wedge 1)](u-1)^+\, dx\\
&\hspace{1cm} -\int_\Omega (u\wedge 1)[c\cdot\nabla (u-1)^+]\, dx-\int_\Omega d(u\wedge 1) (u-1)^+\, dx+\lambda \int_\Omega (u\wedge 1) (u-1)^+\, dx\\
& \hspace{2cm} + \frac{\kappa}{4}\int_\Omega|\nabla u_2|^2\, dx\\
& \hspace{3cm} -\int_\Omega a\nabla u_0\cdot \nabla u_0\, dx-\int_\Omega (b\cdot\nabla u_0)u_0\, dx-\int_\Omega u_0(c\cdot\nabla u_0)\, dx
-\int_\Omega du_0^2\, dx\\
&\hspace{4cm} +\Big(\lambda+ \frac{4\|b-c\|_{L^\infty(\Omega)}^2}{\kappa}\Big)\int_\Omega |u_0|^2\, dx\\
& \leqslant -\int_\Omega c\cdot\nabla (u-1)^+\, dx-\int_\Omega d(u-1)^+\, dx+\lambda \int_\Omega (u-1)^+\, dx \\
&\hspace{1cm}+\frac{\kappa}{4}\int_\Omega|\nabla u_2|^2\, dx\\
&\hspace{2cm} -\int_\Omega a\nabla u_0\cdot \nabla u_0\, dx-\int_\Omega ((b+c)\cdot\nabla u_0)u_0\, dx
-\int_\Omega d|u_0|^2\, dx+\lambda_1^\text{D}(\mathfrak{a}) \int_\Omega |u_0|^2\, dx.
\end{align*}
Since ${\rm div}\, c =0$, $\lambda\leqslant d$ and $c\cdot \nu\geqslant 0$ we have
\[\int_\Omega (-c\cdot \nabla \varphi -d\varphi +\lambda\varphi)\, dx 
=\int_\Omega ({\rm div}\, c-d+\lambda) \varphi\, dx -\int_{\partial\Omega}c\cdot \nu \varphi\, d\sigma \leqslant 0\] 
for all $\varphi\in C^1(\overline{\Omega})_+$, which is dense in $H^1(\Omega)_+$. The conclusion follows.
\end{proof}

Next, we prove Theorem \ref{thm:main}\ref{item:semigroup-DtN-irreducible}. We follow the arguments in \cite[Theorem 4.2]{arendt2012friedlander} 
in order to transfer the irreducibility from $A_\beta$ to $D_\lambda^\mathscr{A}$.

\begin{proof}[Proof of Theorem \ref{thm:main}\ref{item:semigroup-DtN-irreducible}]
Let $\Gamma_1\subset \partial\Omega$ be a Borel set with non-zero measure, put $\Gamma_0:=\partial\Omega \backslash \Gamma_1$, and suppose that 
$L^2(\Gamma_1)=\{\varphi\in L^2(\partial\Omega):\varphi|_{\Gamma_0}=0\ \text{a.e.}\}$ is invariant under $e^{-tD_\lambda^\mathscr{A}}$. The restriction 
$e^{-tD_\lambda^\mathscr{A}}|_{L^2(\Gamma_1)}$ is positive and its generator has compact resolvent (since the embedding 
$H^1(\Omega)\hookrightarrow L^2(\Omega)$ is compact). By the Krein-Rutman theorem the first eigenvalue, say 
$-\beta_1$, of its generator admits an eigenfunction $0<\varphi\in L^2(\Gamma_1)$, which turns to be also a positive eigenfunction of 
$D_\lambda^\mathscr{A}$, that is, $D_\lambda^\mathscr{A}\varphi =-\beta_1\varphi$. By definition, there exists $u\in H^1(\Omega)$ such that 
$\mathscr{A}u=\lambda u$, $u|_{\partial\Omega}=\varphi$ and $\partial_\nu u=-\beta_1 \varphi$; thus, 
$$
-\lambda\int uv\, dx+\frak{a}(u,v)=-\beta_1\int_{\partial\Omega}uv\, d\sigma \ \ \ (v\in H^1(\Omega)). 
$$
Since $u|_{\partial\Omega}=\varphi\geqslant 0$ we have $u^-\in H^1_0(\Omega)$. Inserting $v=u^-$ into the above identity, we obtain
\begin{align}
\lambda\int_\Omega |u^-|^2\, dx&=\lambda\int_\Omega uu^-\, dx=\mathfrak{a}(u,u^-)=
\mathfrak{a}(u^-)\geqslant \lambda_1^\text{D}(\mathfrak{a})\int_\Omega |u^-|^2\, dx
\end{align}
where the last inequality follows from the definition in \eqref{eq:definition-lambda1-AD}. Since $\lambda<\lambda_1^\text{D}(\mathfrak{a})$, 
we must have $u^-=0$, that is, $u\geqslant 0$. By Proposition \ref{prop:spectrum-Robin-DtN}, $A_{\beta_1} u=\lambda u$ and then, from 
the Krein-Rutman theorem, we infer that $\lambda=\lambda_1(A_{\beta_1})$, the first eigenvalue of $A_{\beta_1}$. 

Now, choose a number $\beta_0<\beta_1$ and consider the function 
$\beta\in L^\infty(\partial\Omega)$ given by $\beta=\beta_0{\bf 1}_{\Gamma_0}+\beta_1{\bf 1}_{\Gamma_1}$. Since $u|_{\Gamma_0}=0$ we have
\begin{align}
\mathfrak{a}^\beta(u,v)&=\mathfrak{a}(u,v)+\int_{\partial\Omega}\beta uv\,d\sigma=\mathfrak{a}(u,v)+\beta_1\int_{\partial\Omega} uv\,d\sigma
=\lambda\int_{\Omega} uv\,dx
\end{align}
for all $v\in H^1(\Omega)$, which implies that $u\in \mathscr{D}(A_\beta)$ and $A_\beta u=\lambda u$. Again, the Krein-Rutman theorem allows us to infer 
that $\lambda=\lambda_1(A_{\beta})$, the first eigenvalue of $A_\beta$. We have thus shown that $\lambda_1(A_{\beta_1})=\lambda_1(A_{\beta})$; since 
$\beta\leqslant \beta_1$, 
it follows that 
$e^{-tA_{\beta_1}}\leqslant e^{-tA_{\beta}}$ and then, by Proposition \ref{proposition:first_eigenvalue-characterize-operators}, 
that $A_{\beta}=A_{\beta_1}$. Therefore, $\mathfrak{a}^{\beta}=\mathfrak{a}^{\beta_1}$; in 
particular, $\mathfrak{a}^{\beta}(v)=\mathfrak{a}^{\beta_1}(v)$ for all $v\in H^1(\Omega)$, that is,
\[
\int_{\Gamma_0}\beta_0 v^2\,d\sigma+\int_{\Gamma_1}\beta_1 v^2\,d\sigma=\int_{\partial\Omega}\beta_1 v^2\,d\sigma
\]
for all $v\in H^1(\Omega)$. Hence $\int_{\Gamma_0} v^2\,d\sigma=0$ for all $v\in H^1(\Omega)$. Then, by the Stone-Weierstrass 
theorem, $\int_{\partial\Omega} \varphi^2{\bf 1}_{\Gamma_0}\,d\sigma=0$ for all $\varphi\in C(\partial\Omega)$, which implies that the 
Borel measure ${\bf 1}_{\Gamma_0}\,d\sigma$ is zero; this is the same as saying that $\sigma(\Gamma_0)=0$.
\end{proof}

Finally, we prove Theorem \ref{thm:main}\ref{item:semigroup-DtN-domination}. To this end we observe that, alternatively, the operator 
$D_\lambda^\mathscr{A}$ is also associated with an embedded form, namely, the form $\mathfrak{b}_\lambda$ with 
domain $\mathscr{D}(\mathfrak{b}_\lambda)=j(H^1(\Omega))=j(V_j(\mathfrak{a}_\lambda))$ given by
$$
\mathfrak{b}_\lambda(j(u),j(v)):=\mathfrak{a}_\lambda(u,v)\ \ \ (u,v\in V_j(\mathfrak{a}_\lambda)).
$$
We actually prove below a slightly stronger statement, namely, that the domination property 
$0\leqslant e^{-tD_{\lambda_2}^{\mathscr{A}_2}}\leqslant e^{-tD_{\lambda_1}^{\mathscr{A}_1}}$ holds whenever 
$\lambda_i<\kappa\lambda_1^\text{D}-\|d_i^-\|_{L^\infty(\Omega)}$ ($i=1,2$), $\lambda_2\leqslant \lambda_1$ and 
$d_2\geqslant d_1$; this generalizes \cite[Theorem 2.4]{ter2014analysis}. To make the dependence on $d$ explicit, we write $\mathfrak{a}_\lambda^d$ 
instead of $\mathfrak{a}_\lambda$ (and similarly for $\mathfrak{b}_\lambda$). By $e^{-tD_{\lambda_i}^{\mathscr{A}_i}}$ we mean the 
Dirichlet-to-Neumann semigroup with respect to $\mathfrak{a}_{\lambda_i}^{d_i}$ for $i=1,2$.

\begin{proof}[Proof of Theorem \ref{thm:main}\ref{item:semigroup-DtN-domination}]
By \cite[Theorem 2.24]{ouhabaz2005analysis} it is enough to show that 
$\mathfrak{b}_{\lambda_2}^{d_2}(\varphi,\psi)\geqslant \mathfrak{b}_{\lambda_1}^{d_1}(\varphi,\psi)$ for all 
$0\leqslant \varphi,\psi \in j(H^1(\Omega))=j(V_j(\mathfrak{a}_{\lambda_1}^{d_1}))=j(V_j(\mathfrak{a}_{\lambda_2}^{d_2}))$. 

Let $\varphi,\psi$ be as above. There exists $u_1,v_1\in V_j(\mathfrak{a}_{\lambda_1}^{d_1})$ and $u_2,v_2\in V_j(\mathfrak{a}_{\lambda_2}^{d_2})$ 
such that $\varphi=j(u_1)=j(u_2)$ and $\psi=j(v_1)=j(v_2)$. Since $u_1-u_2\in H_0^1(\Omega)$, it follows as in the proof of 
\ref{item:semigroup-DtN-positive}, and by taking into account that we are assuming $b=c$, that 
$$\mathfrak{a}_{\lambda_2}^{d_2}(u_2-u_1,v_2)=-2\Big(\int_\Omega ((b-c)\cdot\nabla v_2)(u_2-u_1)\, dx =0,$$
that is, $\mathfrak{a}_{\lambda_2}^{d_2}(u_2,v_2)=\mathfrak{a}_{\lambda_2}^{d_2}(u_1,v_2)$. Clearly, we have 
$\mathfrak{a}_{\lambda_1}^{d_1}(u_1,v_2)=\mathfrak{a}_{\lambda_1}^{d_1}(u_1,v_1)$. Moreover, as in the proof of \ref{item:semigroup-DtN-irreducible}, we 
can show that $u_1,v_2\geqslant 0$. Thus,
\begin{align}
\mathfrak{b}_{\lambda_2}^{d_2}(\varphi,\psi)&=\mathfrak{a}_{\lambda_2}^{d_2}(u_2,v_2)\\
&=\mathfrak{a}_{\lambda_1}^{d_1}(u_1,v_1)+\int_\Omega (d_2-d_1)u_1v_2\, dx+\int_\Omega (\lambda_1-\lambda_2)u_1v_2\, dx\\
&\geqslant \mathfrak{b}_{\lambda_1}^{d_1}(\varphi,\psi) . 
\end{align}
\end{proof}

Domination properties of the semigroups $e^{-tD_{\lambda}^\mathscr{A}}$ in the nonself-adjoint case $b\neq c$ seem to be more difficult, at 
least under the present hypotheses and by using the methods employed in this paper. Two deeper problems would be (i) to investigate 
the graph $D_\lambda^\mathscr{A}$ when $\lambda\in \sigma(A_\text{D})$ along the lines of \cite{arendt2014dirichlet} 
and \cite{behrndt2015dirichlet} and (ii) to investigate the operator (or graph) $D_\lambda^\mathscr{A}$ on 
rough domains in the spirit of e.g. \cite{arendt2011dirichlet} or \cite{gesztesy2011description}. These questions 
will be investigated elsewhere.

\subsection*{Acknowledgment}

The authors thank the anonymous referee for several suggestions which improved the presentation in many respects.

\bibliographystyle{plain}
\bibliography{biblio-DtN}

\end{document}